\documentclass[12pt]{amsart}
\usepackage{amsmath,amsthm,amssymb,graphicx}
\usepackage{longtable}

\newcommand\mymat[4]{
{\left(
\begin{smallmatrix}#1&#2\\#3&#4\end{smallmatrix}
\right)}}

\def\Sp{\operatorname{Sp}}

\def\SL{\operatorname{SL}}

\def\Grit{\operatorname{Grit}}
\def\Borch{\operatorname{Borch}}
\def\Hum{\operatorname{Hum}}
\def\IM{\operatorname{Im}}

\def\LA{{\mathcal L\/}}
\def\tr{\operatorname{tr}}

\def\Z{{\mathbb Z}}

\def\C{{\mathbb C}}
\def\N{{\mathbb N}}
\def\Q{{\mathbb Q}}
\def\Pj{{\mathbb P}}

\def\cal{\mathcal}

\let\smtwomat\mymat

\theoremstyle{plain}
\newtheorem{thm}{Theorem}[section]
\newtheorem{lm}[thm]{Lemma}

\newtheorem{cor}[thm]{Corollary}

\newtheorem{df}[thm]{Definition}

\theoremstyle{definition}

\def\cusp{{\text{\rm cusp}}}
\def\ord{{\text{\rm ord}}}
\def\TB{\mathop{\text{\rm TB}}}
\def\BTB{\mathop{\text{\rm BTB}}}

\def\Half{{\cal H}}

\def\<{\langle}
\def\>{\rangle}

\def\wh{^\text{{\rm w.h.}}}
\def\weak{^\text{{\rm weak}}}
\def\mero{^\text{{\rm mero}}}
\def\inv{^{-1}}

\newcommand\ldL{\llcorner}
\newcommand\rdL{\lrcorner}

\def\Coeff{\operatorname{Coeff}}

\def\Fam#1{{\cal F}_#1}

\def\AA{{\mathcal A}}
\def\End{\operatorname{End}}
\def\new{{\text{\rm new}}}

\begin{document}
\title[Antisymmetric Forms]
{Antisymmetric Paramodular Forms of Weights 2 and 3}

\author[V. Gritsenko]{Valery Gritsenko}
\address{Laboratoire Paul Painlev\'e, Universit\'e Lille 1,
59655 Villeneuve d'Ascq Cedex France and  IUF; {\it current address:}
National Research University Higher School of Economics, Moscow}
\email{Valery.Gritsenko@math.univ-lille1.fr}

\author[C. Poor]{Cris Poor}
\address{Dept{.} of Mathematics, Fordham University, Bronx, NY 10458, USA}
\email{poor@fordham.edu}

\author[D. Yuen]{David S. Yuen}
\address{Department of Mathematics and Computer Science, Lake Forest College, 555 N. Sheridan Rd., Lake Forest, IL 60045, USA}
\email{yuen@lakeforest.edu}

\date{\today}

\begin{abstract}
We define an algebraic set in $23$~dimensional projective space whose $\Q$-rational points 
correspond to meromorphic, antisymmetric, paramodular Borcherds products.  We know two lines inside this 
algebraic set. Some rational points on these lines give holomorphic Borcherds products and thus 
construct examples of Siegel modular forms on degree two paramodular groups.  
Weight~$3$ examples provide antisymmetric canonical differential forms on Siegel modular 
threefolds.  Weight~$2$ is the minimal weight and these examples, via the Paramodular Conjecture, give evidence for  the modularity 
of some  rank one abelian surfaces defined over~$\Q$.  
\end{abstract}


\maketitle

\section{Introduction}

This article formulates the construction of meromorphic Borcherds products as a diophantine problem. 
We discover two infinite families of solutions and produce interesting examples of 
holomorphic antisymmetric paramodular forms of weights~$2$ and~$3$.  
The first author constructed paramodular cusp forms of low weight and used them to study the 
geometry of the moduli spaces of abelian surfaces and Kummer surfaces \cite{GritArith}, \cite{Grit2}, \cite{Grit3}.  
These paramodular forms were constructed by lifting and so were symmetric.  
In \cite{Grit3}, with geometric applications in mind, the first author posed the problem 
of constructing antisymmetric paramodular cusp forms of low weight.

Further motivation to construct paramodular cusp forms of low weight 
comes from questions of modularity.  
Certain 
cusp forms of weight~$2$ and paramodular level~$N$ conjecturally 
correspond to abelian surfaces of conductor~$N$ defined over~$\Q$ 
whose endomorphisms defined over~$\Q$ are trivial, 
as made precise in the Paramodular Conjecture 
of Brumer and Kramer \cite{BK}.  
Also, nonlift 
paramodular cusp forms of weight~$3$  conjecturally occur as cohomology 
classes in $H^5\left( \Gamma_0(N); \C \right)$,  as 
studied by Ash, Gunnells and McConnell \cite{AGM}.  

In this article, we use Borcherds products to construct examples of 
antisymmetric paramodular cusp forms of weights~$2$ and~$3$ 
with applications to both geometry and modularity.  Weight two is the minimal weight, and 
weight three is the canonical weight.  Our method relies 
heavily on the {\sl theta blocks\/} introduced by Gritsenko, Skoruppa and 
Zagier in \cite{GSZ}.  We were led to these constructions by the Paramodular Conjecture and so 
we devote some space to it here.  
\smallskip

\noindent
{\bf Paramodular Conjecture\/}  
{\rm (A. Brumer, K. Kramer, 2010.)\/}  \newline
{\it Let $N \in \N$.  
There is a bijection between:  \newline 
$(1)$  isogeny classes of abelian surfaces~$\AA$ defined over~$\Q$ 
with conductor~$N$ and $\End_{\Q}(\AA)=\Z$, and \newline
$(2)$  lines of Hecke eigenforms $f \in S_2\left( K(N) \right)^\new$ 
that are not Gritsenko lifts from $J_{2,N}^\cusp$ and 
whose eigenvalues are rational.  

In this correspondence, 
$ L(\AA,s, \text{\rm Hasse-Weil})= L(f,s, \text{\rm spin})$. }
\smallskip

In \cite{PoorYuenPara}, intensive computations proved that, for primes~$p<600$, 
the space $S_2\left( K(p) \right)$ is spanned by 
Gritsenko lifts with the possible exceptions of 
$p \in \{ 277, 349, 353, 389, 461, 523, 587 \}$.   
These computations are consistent with the Paramodular Conjecture in two ways.  
First, Brumer and Kramer proved, among other results \cite{BK},  
that for primes~$p<600$ no abelian surfaces over~$\Q$ with conductor~$p$ 
can exist apart from the $p$ on this list.  
Second, an abelian surface defined over~$\Q$ with conductor~$p$ is known 
for every prime on this list.  In fact, for $p=587$ two isogeny classes are known, 
one $\AA_{587}^{+}$ with even rank zero and one $\AA_{587}^{-}$ with 
odd rank one.  Still, on the automorphic side, there remains the problem of giving a construction to 
prove the existence of weight two paramodular cusp forms that are not Gritsenko lifts.  
Only in one case, $p=277$, was the existence of a  weight two paramodular 
nonlift eigenform proven.  
This article  constructs a  weight two nonlift paramodular eigenform in $S_2\left( K(587) \right)^{-}$ and thereby 
produces additional evidence for the Paramodular Conjecture.  This 
construction is our motivating example, see section~4.  

We also construct a nonlift paramodular eigenform in $S_3\left( K(167) \right)^{+}$, 
which experimentally seems to be the first weight three form all of whose paramodular 
Atkin-Lehner signs are~$+1$.  This form induces a canonical differential form on a moduli 
space of polarized Kummer surfaces, $K(167)^{+}\backslash\Half_2$, thereby proving that it is not 
rational or unirational.  The final section gives a complete list of the  examples we found 
by solving diophantine equations.   

{\Small The first author  was partially supported by the Russian Academic Excellence Project `5-100' (research in arithmetic) and by RSF grant, project 14-21-00053 (geometric applications).}

\section{Background}

We follow \cite{freitag83} for the theory of Sigel modular forms.  
For a group $\Gamma$ 
commensurable with $\Gamma_n = \Sp_n(\Z)$, let 
$M_{k}(\Gamma, \chi)$ denote the $\C$-vector space of Siegel modular forms 
of weight~$k$ and character~$\chi$, and $S_{k}(\Gamma, \chi)$ the subspace of cusp forms.  The basic 
factor of automorphy is defined by $\mu(\mymat{A}BCD, \Omega)= \det(C \Omega+D)$.  
The space of meromorphic functions transforming by $\mu^k \chi$ for $\Gamma$ is 
denoted by $M_{k}^{ \text{mero} }(\Gamma, \chi)$.   
We will be concerned with degree  $n=2$ and the {\it paramodular\/} group 
of paramodular  level~$N$:  
$$
K(N)= 
\begin{pmatrix} 
*  &  N*  &  *  &  *  \\
*  &  *  &  *  &  */N  \\
*  &  N*  &  *  &  *  \\
N*  &  N*  &  N*  &  *  
\end{pmatrix} 
\cap \Sp_2(\Q), 
\quad \text{ $ * \in \Z$. }  
$$

For any natural number $N>1$, the paramodular group $K(N)$ has a normalizing involution   
$\mu_N$ given by 
$\mu_N = \mymat{F_N^{*}}{0}{0}{F_N}$, where 
$F_N =\frac{1}{ \sqrt{N} } \mymat{0}{1}{-N}{0}$ is the Fricke involution, 
and we also use the group 
$K(N)^{+}$ generated by $ K(N)$ and $\mu_N $.  
We have the direct sum 
$M_k( K(N) ) = M_k(K(N))^{+} \oplus M_k( K(N))^{-}$ 
into plus and minus $\mu_N$-eigenspaces.  
This involution and the Witt map can be used to prove that 
$M_2\left( K(p) \right) = S_2\left( K(p) \right)$ when $p$ is prime, see 
\cite{GPY}.  
We can also consider the group $K(N)^{*}$, generated by $K(N)$ and all the 
paramodular Atkin-Lehner involutions, see \cite{GritHulek0}; 
like $\mu_N$, these have an elliptic Atkin-Lehner involution in the lower right 
and the transpose inverse in the upper left.  
The moduli space of Kummer 
surfaces with polarization type $(1,N)$ is isomorphic to $K(N)^{*} \backslash \Half_2$.  One 
can use the paramodular Atkin-Lehner operators, the Witt map, and the cusp structure of $K(N)$ to 
show that $M_2\left( K(N) \right) = S_2\left( K(N) \right)$ for squarefree~$N$.  
\smallskip

We define Jacobi forms with reference to the subgroup
$$
P_{2,1}(\Z)= 
\begin{pmatrix} 
*  &  0  &  *  &  *  \\
*  &  *  &  *  &  *  \\
*  &  0  &  *  &  *  \\
0  &  0  &  0  &  *  
\end{pmatrix} 
\cap \Sp_2(\Z), 
\quad \text{ $ * \in \Z$.}  
$$
The group $P_{2,1}(\Z)$ is generated by a Heisenberg group~$H(\Z)$, 
a copy of~$\SL_2(\Z)$, and by~$\{I_4,-I_4\}$. The general element~$h$ 
of~$H(\Z)$ and $\sigma$ of~$\SL_2(\Z)$ have the form
$$
h = 
\begin{pmatrix} 
1  &  0  &  0  &  v  \\
\lambda  &  1  &  v  &  \kappa  \\
0  &  0  &  1  &  -\lambda  \\
0  &  0  &  0  &   1
\end{pmatrix};
\quad 
\sigma= 
\begin{pmatrix} 
a  &  0  &  b  &  0  \\
0  &  1  &  0  &  0  \\
c  &  0  &  d  &  0  \\
0  &  0  &  0  &   1
\end{pmatrix}, 
$$
for $ \lambda, v, \kappa \in \Z$,  and for $\mymat{a}bcd \in \SL_2(\Z)$.   
A  character $v_H: H(\Z) \to \{ \pm 1 \}$ is 
given by $v_H(h) = (-1)^{\lambda v + \lambda + v + \kappa}$ 
and extends to a character on  $P_{2,1}(\Z)$ trivial on $\SL_2(\Z)$. 
Similarly, the multiplier~$\epsilon$ of the Dedekind Eta function extends to a 
multiplier on   $P_{2,1}(\Z)$ trivial on $H(\Z)$.  
  
The following definition of Jacobi forms,  see \cite{GritNiku98PartII},  
is equivalent to the usual one  \cite{EZ}. 
For $a,b, 2k, 2m \in \Z$, 
consider holomorphic $\phi: \Half_1 \times \C \to \C$, 
such that the modified function ${\tilde \phi}: \Half_2 \to \C$,  
given by ${\tilde \phi}\mymat{\tau}{z}{z}{\omega} = \phi(\tau, z) e( m \omega)$, 
transforms by the factor of automorphy $\mu^k \epsilon^a v_H^b$ for $P_{2,1}(\Z)$.  
We always select holomorphic branches of roots that are positive on the purely imaginary elements of the Siegel half space.  
We necessarily have $2k \equiv a \mod 2$, $2m \equiv b \mod 2$,   and $m \ge 0$ for nontrivial~$\phi$.  
Such $\phi$ have Fourier expansions 
$ \phi(\tau, z) = \sum_{ n,r \in \Q } c(n,r; \phi) q^n \zeta^r$, 
for $q=e(\tau)$ and $\zeta= e(z)$.  
We write $\phi \in J_{k,m}\wh( \epsilon^a v_H^b)$ if, 
additionally, the support of $\phi$ has $n$ bounded from below, 
and call such forms {\it weakly holomorphic.\/}   We write 
 $\phi \in J_{k,m}\weak( \epsilon^a v_H^b)$ if the support of 
 $\phi$ satisfies $n \ge 0$; 
 $\phi \in J_{k,m}( \epsilon^a v_H^b)$ if $4mn-r^2 \ge 0$; 
 $\phi \in J_{k,m}^\cusp( \epsilon^a v_H^b)$ if $4mn-r^2 > 0$.  
We also use the notation 
$ J_{k,m}^\cusp( \nu)=\{ \phi \in  J_{k,m}^\cusp: \ord_q\phi \ge \nu \} $ 
and denote the elements of $ J_{k,m}\wh$ whose Fourier coefficients are integral 
by  $J_{k,m}\wh( \Z)$.  

 For basic examples of Jacobi forms we make use of  
the Dedekind Eta function 
$
\eta(\tau)=q^{\frac{1}{24}} 
\prod_{n \in \N} (1-q^n) 
$
and the odd Jacobi theta function:  
\begin{align*}
\vartheta(\tau, z) 
&=q^{\frac18}\left( \zeta^{\frac12}- \zeta^{-\frac12} \right) 
\prod_{j \in \N}(1 - q^j\zeta)(1-q^j\zeta^{-1})
(1-q^j).
\end{align*}
We have $\vartheta \in J_{\frac12,\frac12}^\cusp( \epsilon^3 v_H)$, 
$ \eta \in J_{\frac12,0}^\cusp( \epsilon)$ and 
$\vartheta_{\ell} \in J_{\frac12,\frac12 \ell^2}^\cusp( \epsilon^3 v_H^\ell)$, 
where $\vartheta_{\ell}(\tau,z) = \vartheta(\tau, \ell z)$ and $\ell \in \N$, 
compare \cite{GritNiku98PartII}.  
\smallskip

The Fourier-Jacobi expansion of a paramodular form is an important connection 
between Jacobi and paramodular forms.  
For $\Omega=\mymat{\tau}zz{\omega} \in \Half_2$, and $ f \in M_k\left( K(N) \right)^{\epsilon}$, 
write the Fourier-Jacobi expansion of $f$ as 
$$
f(\Omega) = \sum_{m=0}^{\infty} 
\phi_{Nm}(\tau,z) e(Nm\omega).
$$
Each Fourier-Jacobi coefficient is a Jacobi form, 
$\phi_{Nm} \in J_{k,Nm}$, and these Fourier coefficients satisfy \cite{IPY} the 
{\it involution condition:\/}
\begin{equation}
\label{invcond}
c(n,r;\phi_{Nm}) = (-1)^k \epsilon \, c(m,r;\phi_{Nn}).
\end{equation}
Paramodular forms with $ (-1)^k \epsilon=+1$ are called {\it symmetric\/}, those 
with  $ (-1)^k \epsilon=-1$ are called {\it antisymmetric.\/}

We let 
$\epsilon^a \times v_H^b:K(N)^{+} \to e(\frac{1}{12}\Z)$ 
denote the unique character, if it exists, 
whose restriction to $P_{2,1}(\Z)$ is $\epsilon^a  v_H^b$, 
and whose value on $\mu_N$ is~$1$.  
For $a,b \in \Z$, the character  exists 
precisely when there is a $j \in \Z$ such that 
$a \equiv j \frac{24}{\gcd(2N,12)} \mod 24$ and $b \equiv j \frac{2N}{\gcd(2N,12)} \mod 2,  $ 
see \cite{GritHulek}.

\smallskip

Let $ \phi \in J_{k,t}\wh$ be a weakly holomorphic Jacobi form.  
Recall the level raising Hecke operators 
$V_{m}: J_{k,t}\wh \to J_{k,tm}\wh$ from \cite{EZ}, page 41.  
These operators 
have the following action on Fourier coefficients: 
$$
c\left( n,r;\phi | V_m \right)= 
\sum_{ d \in \N: \, d | (n,r,m) } 
d^{k-1} c\left( \frac{nm}{d^2},\, \frac{r}{d}; \phi \right).  
$$
Given any $\phi \in J_{k,t}\wh$, 
we may consider the following series  
$$
\Grit(\phi)\mymat{\tau}zz{\omega} = 
\delta(k)c(0,0;\phi)G_k(\tau)+
\sum_{m \in \N} 
\left( \phi | V_m \right)(\tau,z) e( mt \omega) 
$$
where  $\delta(k)=1$ for even  $k\ge 4$ and  $\delta(k)=0$ for all other $k$,  and  
$G_k(\tau)=\frac12  \zeta(1-k)+
\Sigma_{n\ge 1}\sigma_{k-1}(n)e(\tau)$
is the Eisenstein series of weight~$k$.  
\begin{thm} 
{\rm  (\cite{GritArith}, \cite{Grit2})\/} 
For $\phi \in J_{k,t}$, the series $\Grit(\phi)$ converges on $\Half_2$ and defines a holomorphic 
function $\Grit(\phi): \Half_2 \to \C$ 
that 
is an element of $M_k\left( K(t)\right)^{\epsilon}$ with $\epsilon=(-1)^k$.
This is a cusp form if  $\phi\in J^{\rm cusp}_{k,t}$.  
\end{thm}

The {\it Gritsenko lift,\/} $\Grit(\phi)$, of the Jacobi form~$\phi$ 
defines a linear map 
$\Grit: J_{k,t} \to M_k\left( K(t) \right)^{\epsilon}$.  
Gritsenko lifts are hence symmetric.  
A different type of lifting construction is due to Borcherds (see \cite{B}) 
via his theory of multivariable  infinite products 
on orthogonal groups.
The divisor of a Borcherds Product is supported on rational quadratic divisors.
In the case of the Siegel upper half plane of degree two, these 
rational quadratic divisors are the 
Humbert modular surfaces.  
\begin{df}
Let $N \in \N$.  
For $n_o,r_o,m_o \in \Z$ with $m_o \ge 0$ and $\gcd(n_o,r_o,m_o)=1$,  
set $T_o = \mymat{n_o}{r_o/2}{r_o/2}{Nm_o}$ when $\det(T_o) < 0$.
We call
$$
\Hum(T_o)= K(N)^{+} \{ \Omega \in \Half_2:\, \tr( \Omega\,  T_o )=0 \} \subseteq K(N)^{+} \backslash \Half_2.
$$  
a {\sl Humbert  surface\/}.  
\end{df}
A Humbert surface $\Hum(T_o)$ 
only depends upon two pieces of data  \cite{GritHulek0}: 
the discriminant $D=r_o^2-4Nm_on_o$ 
and $r_o \mod 2N$.  

\begin{thm} 
\label{BP}
{\rm (\cite{GN1}, \cite{GritNiku98PartII}, \cite{Grit24})}   
Let $N, N_o \in \N$.    
Let $\Psi \in J_{0,N}\wh$ be a weakly holomorphic Jacobi form with Fourier expansion
$$
\Psi(\tau,z) = \sum_{n,r \in \Z: \, n \ge -N_o } c\left( n,r \right) q^n \zeta^r
$$
and $c(n,r) \in \Z$ for $4Nn-r^2 \le 0$.  
Then we have  $c(n,r) \in \Z$ for all $n,r \in \Z$.  
We set
\begin{align*}
&24A= \sum_{\ell \in \Z} c(0,\ell); \quad 
2 B = \sum_{\ell \in \N} \ell c(0,\ell); \quad 
4 C = \sum_{\ell \in \Z} \ell^2 c(0,\ell);   \\
&D_0 = \sum_{n \in \Z: \, n <0} \sigma_0(-n) c(n,0); \ 
k= \frac12 c(0,0); \ 
\chi = \epsilon^{24A} \times v_H^{2B}.
\end{align*}
There is a function $\Borch(\Psi) \in M_k\mero\left( K(N), \chi \right)$ 
whose divisor in $K(N)^{+} \backslash \Half_2$ consists of Humbert surfaces 
$\Hum(T_o)$ 
for $T_o = \mymat{n_o}{r_o/2}{r_o/2}{Nm_o}$ with $\gcd(n_o,r_o,m_o)=1$ and $m_o \ge 0$.  
The multiplicity of $\Borch(\Psi) $ on $\Hum(T_o)$ is 
$\sum_{n \in \N} c(n^2n_om_o, nr_o)$.  
 We have,  
$$
\Borch(\Psi)(F_N'{\Omega}F_N)=(-1)^{k+D_0} \Borch(\Psi)(\Omega), 
\text{ for $\Omega \in \Half_2$. }  
$$ 
For sufficiently large  $\lambda$, for $\Omega=\mymat{\tau}zz{\omega} \in \Half_2$ 
and $q=e(\tau)$, $\zeta=e(z)$, $\xi=e(\omega)$, 
the following product converges on $\{\Omega \in \Half_2: \IM \Omega > \lambda I_2 \}$:  
$$
\Borch(\Psi)(\Omega){=}
q^A \zeta^B \xi^C 
\prod_{\substack{ n,r,m \in \Z:\, m \ge 0, \text{\rm\ if $m=0$ then $n \ge 0$} \\  \text{\rm and if $m=n=0$ then $r < 0$. } }}
\left( 1-q^n \zeta^r \xi^{Nm} \right)^{ c(nm,r) }
$$
and is on $\{\Omega \in \Half_2: \IM \Omega > \lambda I_2 \}$ 
a rearrangement of
\begin{equation}
\label{B3}
\Borch(\Psi)= 
\left(  \eta^{c(0,0)}  \prod_{ \ell \in \N} \left( {{\tilde \vartheta}_{\ell}}/{\eta} \right)^{c(0,\ell)} \right) 
\exp\left( - \Grit({ \Psi })\right).  
\end{equation}
\end{thm}

\section{Theta Blocks and Inflation}

In the previous section, theta blocks occured naturally as the leading Fourier-Jacobi 
coefficient of a Borcherds product, see \cite{GSZ} for a full exposition.  
A theta block is a meromorphic function $\TB_f: \Half \times \C \to \C$, 
$$
{\TB}_f = \eta^{f(0)}\, 
\prod_{\ell \in \N} 
\left( {\vartheta_{\ell}}/{\eta} \right)^{f(\ell)}, 
$$
for some even function $f: \Z \to \Z$ of finite support.  We see that 
$\TB_f \in J_{k,m}\mero(\chi)$, 
where the weight~$k$, index~$m$ and character~$\chi$ are given by
\begin{align*}
k  &= \frac12 f(0);  \quad
4m  = \sum_{\ell \in \Z} \ell^2 f(\ell); \,   \\
\chi &= \epsilon_{\eta}^{24A} v_H^{2B};  \quad
24A = \sum_{\ell \in \Z} f(\ell); \quad 2B = \sum_{\ell \in \N} \ell f(\ell).  
\end{align*}
We call the Laurent polynomial 
$$ 
\sum_{\ell \in \Z} f(\ell) \zeta^{\ell}= 
f(0) + \sum_{\ell\in\N} f(\ell)(\zeta^\ell+\zeta^{-\ell})
$$
the {\em germ} of the theta block $\TB_f$.  Note that the germ determines the theta block.
By the {\it shape\/} of a theta block we simply mean the number of factors of $\eta$ 
and of $\vartheta$ in the numerator and denominator.  
To test whether a weakly holomorphic theta block is a Jacobi form, we need to see whether the 
order function \cite{GSZ} 
$$
\ord( {\TB}_f ; x)= \frac{k}{12} + \frac12 \sum_{\ell \in \N} f(\ell) {\bar B}_2(\ell x)
$$
is nonnegative for $0 \le x \le 1$.  If $\ord( {\TB}_f ; x)>0$  then 
$\TB_f \in J_{k,m}^\cusp(\chi)$.  
Here, ${\bar B}_2(x)= B_2(x-\ldL x \rdL)$ is the periodic extension of the second 
Bernoulli polynomial $B_2(x) =x^2-x+\frac16$.  
If we multiply the defining infinite products together, we obtain
\begin{align*}
{\TB}_f(\tau,z) &= 
q^{A}\, 
\prod_{\ell \in \N} \left( \zeta^{\ell/2} - \zeta^{-\ell/2} \right)^{f(\ell)}\, 
\prod_{n \in \N} \prod_{\ell \in \Z} (1-q^n \zeta^{\ell} )^{f(\ell)}  \\
&=
q^{A}\, 
\zeta^{ B}  
\prod_{\substack{n, \ell \in \Z: \, n \ge 0 \text{ and} \\ \text{ if $n=0$ then $\ell<0$}}} (1-q^n \zeta^{\ell} )^{f(\ell)} .  
\end{align*}
We set $\BTB_f(\zeta)= \prod_{\ell \in \N} \left( \zeta^{\ell/2} - \zeta^{-\ell/2} \right)^{f(\ell)}$ 
and affectionately refer to this as the {\it baby theta block.\/}  The weight and the baby 
theta block determine the theta block.

When all the $f(\ell)$ are nonnegative, 
then the theta functions are all in the numerator and 
we say that ${\TB}_f$ is a 
theta block {\it without theta denominator.\/}  
Theta blocks without theta denominator are automatically weakly holomorphic. 
In computations involving theta blocks without theta denominator 
it is often more convenient 
to describe theta blocks in terms of lists.  Let $T=T_f$ be the list 
consisting of $f(1)$ $1$s followed by $f(2)$ $2$s, $\dots$, etc{.}   
and let $\LA=(-T):[\overbrace{0,0,\dots,0}^{2k}]:T$ be defined as the concatenation of three lists.  
When we are dealing with lists we have $24A = | \LA | $ and we write 
\begin{align*}
{\TB}_k(T)(\tau,z) &= 
\eta^{2k}\prod_{d\in T} \left(\vartheta_{d}/\eta  \right) \\
&=q^{A}\, 
\prod_{d \in T} \left( \zeta^{d/2} - \zeta^{-d/2} \right)\, 
\prod_{n \in \N} \prod_{ e \in \LA} (1-q^n \zeta^{e} )
\end{align*}
and write the baby theta block 
${\BTB}(T)= \prod_{d \in T} \left( \zeta^{d/2} - \zeta^{-d/2} \right)$.  
\smallskip

In order to construct paramodular Borcherds products we need 
a quotient~$\psi$ of a Jacobi form by a theta block of the same weight.  
This weight zero~$\psi$ should have integral Fourier coefficients and 
be weakly holomorphic as opposed to properly meromorphic.  
We now discuss two methods to guarantee that both these  properties hold.  
The first method works for theta blocks without theta denominator 
and can fail otherwise. 
The second method is more general and 
does allow the theta block to have a theta denominator.  
\smallskip

\noindent {\bf Method 1. \/}  
Let $\phi$ be a theta block without theta denominator.  
The quotient $\dfrac{\phi | V_{\ell} }{ \phi }$ has  integral Fourier coefficients and 
is weakly holomorphic, see \cite{GPY}, Corollary~{6.2}, for the case $\ell=2$.  
\smallskip

\noindent {\bf Method 2. \/} 
Suppose that ${\BTB}_f(\zeta)$ divides $ {\BTB}_g(\zeta)$ in $\Z[\zeta^{1/2}, \zeta^{-1/2}]$, then 
the weight zero Jacobi form 
$\psi=\dfrac{{\TB}_g }{{\TB}_f}$ has  integral Fourier coefficients and 
is weakly holomorphic.  
When we have this divisibility, we call $\TB_g$ an {\it inflation\/} of $\TB_f$. \smallskip

As a special case of method two, 
suppose that the theta block is $\phi={\TB}_k(T)$ for some list~$T$ and weight~$k$.  
Now take another list~$E$ of the same length with 
each entry of $T$, in some ordering, dividing the 
corresponding entry of~$E$; in this case, 
we clearly have that ${\BTB}(T)$ divides $ {\BTB}(E)$ in the ring $\Z[\zeta^{1/2}, \zeta^{-1/2}]$ and 
we call the theta block ${\TB}_k(E)$ 
a {\it strict inflation\/} of the theta block $ {\TB}_k(T)$.

\section{A Motivating Example}

In this section we use the theory of Borcherds products to construct a paramodular eigenform 
$f \in S_2\left( K(587) \right)^{-}$ that conjecturally shows the modularity of the 
isogeny class of rank one abelian surfaces defined over~$\Q$ of prime conductor 
$p=587$, the isogeny class represented by the Jacobian ${\mathcal A}_{587}^{-}$ of 
$y^2+ (x^3+x+1)y=-x^3-x^2$.  In \cite{PoorYuenPara} it is proven that $\dim  S_2\left( K(587) \right)^{-} \le 1$ 
and that $L({\mathcal A}_{587}^{-}, s, \text{Hasse-Weil})$ and $L(f,s, \text{spin})$ share the 
same Euler factors at~$2$ and~$3$, if $f$ can be proved to exist.  
The following 
analogy with the elliptic case led to the hope that $f$ is a Borcherds product.
The first odd rank elliptic curve over~$\Q$ has conductor~$37$ and the associated 
modular form in $S_2\left( \Gamma_0(37) \right)^{+}$ corresponds to a Jacobi form in 
$J_{2,37}^\cusp$.  From \cite{GSZ} we know that $J_{2,37}^\cusp$ is generated by the 
theta block $\TB_2(1,1,1,2,2,2,3,3,4,5)$ and thus is given by an infinite product.  
The smallest known prime conductor of an odd rank abelian surface over~$\Q$ is $587$, see \cite{BK}.  
Our Borcherds product construction of~$f$ proves the analogous result that the 
generator $f$ of $ S_2\left( K(587) \right)^{-}$ is a multivariable infinite product. 
This Borcherds product representation of the eigenform~$f$ will assist the computation of its 
eigenvalues.

We describe the experimental process that led to the rigorous construction of $f$ as a Borcherds product.  
Both formula~{(\ref{B3})} and the involution condition~{(\ref{invcond})} are important tools in the experimental process. 
Suppose that we have an antisymmetric Borcherds product $f \in S_k\left( K(p) \right)^{\epsilon}$ 
with Fourier-Jacobi expansion
$$
f= \phi_{p}\, e(p\omega) + \phi_{2p}\, e(2p\omega) + \cdots 
$$
The antisymmetry of $f$ implies that $\phi_{p} \in J_{k,p}^\cusp(2)$ 
because $c(1,r; \phi_{p}) = - c(1,r; \phi_{p})$.  
If $f$ is a Borcherds product then $\phi_{p}$ is a theta block;  indeed, 
$J_{2, 587}^\cusp(2)$ is spanned by the single theta block 
$$
\phi_{587}={\TB}_2(
1, 1, 2, 2, 2, 3, 3, 4, 4, 5, 5, 6, 6, 7, 8, 8, 9, 10, 11, 12, 13, 14
).
$$
As a Borcherds product, $f$ is determined by its first two Fourier-Jacobi coefficients 
$\phi_{p}$ and $\phi_{2p}$.  If we write $ \phi_{2p} = - \phi_{p} \vert V_2 + \Xi$, 
we claim this forces $\Xi \in J_{k,2p}^\cusp(2)$.  
The involution condition tells us the $q^1$ and $q^2$ coefficients of $\Xi$.  
In general 
\begin{align*}
c(n,r; \Xi) 
&= c(n,r; \phi_{p}\vert V_2) +c(n,r; \phi_{2p})  \\
&= c(2n,r; \phi_{p}) + 2^{k-1} c(\frac{n}2,\frac{r}2; \phi_{p})  +c(n,r; \phi_{2p}).
\end{align*}
The $q^1$-coefficient of $\Xi$ vanishes by antisymmetry
$$
c(1,r; \Xi) = c(2,r;\phi_p) + c(1,r; \phi_{2p}) 
= - c(1,r; \phi_{2p})  +  c(1,r; \phi_{2p})  =0,   
$$
so that  $\Xi \in J_{k,2p}^\cusp(2)$ as claimed.  
The $q^2$-coefficient of $\Xi$ is the $q^4$-coefficient of $\phi_p$: 
\begin{align*}
c(2,r; \Xi) 
&= c(4,r; \phi_{p}) + 2^{k-1} c(1,\frac{r}2; \phi_{p})  +c(2,r; \phi_{2p}) =  c(4,r; \phi_{p}).
\end{align*} 
A direct computation reveals that the $q^4$-coefficient of $\phi_{587}$ 
is a baby theta block:  
 $\Coeff(q^4, \phi_{587})=  \Coeff(q^2, \Xi)= $
\begin{align*}
&\BTB(1, 10, 2, 2, 18, 3, 3, 4, 4, 15, 5, 6, 6, 7, 8, 16, 9, 10, 22, 12, 13, 14) .
\end{align*}
So we set 
\begin{align*}
\Xi=&{\TB}_2(1, 10, 2, 2, 18, 3, 3, 4, 4, 15, 5, 6, 6, 7, 8, 16, 9, 10, 22, 12, 13, 14)
\end{align*}
and check that $\Xi \in J_{2,1174}^\cusp(2)$.  
The list for this theta block $\Xi$ has been written so that each entry is visibly an integral multiple 
of the corresponding entry in the theta block $\phi_{587}$; 
thus $\Xi$ is a strict inflation of  $\phi_{587}$.  

We now combine our two methods for constructing weight zero weakly holomorphic 
Jacobi forms with integral Fourier coefficients and define
$$
\psi =\dfrac{\phi_{587}\vert V_2 - \Xi}{\phi_{587}} \in J_{0,587}\wh(\Z).  
$$
Many Fourier coefficients of $\psi = \sum_{n,r} c(n,r; \psi) q^n \zeta^r$ 
can be seen at \cite{PoorYuenWebPara587minus} and we give the singular part of $\psi$ here, 
up to $q^{\ldL  p/4  \rdL } = q^{147}$: 
\begin{equation*}\label{psisingularpart}
\begin{aligned}
&\psi_{\text{\rm sing}} = \frac{1}{q} + 4 + \zeta^{-14} + 
  \zeta^{-13} + \zeta^{-12} {+} 
  \zeta^{-11} + \zeta^{-10} {+} 
  \zeta^{-9} + 2{\zeta^{-8}} {+} 
  \zeta^{-7}+ \\
&{2}{\zeta^{-6}} + 
 {2}{\zeta^{-5}} + 
  {2}{\zeta^{-4}} + 
  {2}{\zeta^{-3}} + 
 {3}{\zeta^{-2}} + 
  {2}{\zeta^{-1}} + 2\,\zeta + 
  3\,\zeta^2 + 2\,\zeta^3 +\\&
 2\,\zeta^4 + 
  2\,\zeta^5 + 2\,\zeta^6 + \zeta^7 + 
  2\,\zeta^8 + \zeta^9 + \zeta^{10} + 
  \zeta^{11} + \zeta^{12} + \zeta^{13} + 
  \zeta^{14} +\\&
 q\,
   \left( \zeta^{-50} + 
     {2}{\zeta^{-49}} + 
     2\,\zeta^{49} + \zeta^{50}
     \right)  + 
  q^2\,\left( \zeta^{-69} + 
     \zeta^{69} \right)  + \\&
  q^3\,\left( \zeta^{-85} + 
     {2}{\zeta^{-84}} + 
     2\,\zeta^{84} + \zeta^{85}
     \right)  + 
  q^4\,\left( \zeta^{-98} + 
     {2}{\zeta^{-97}} + 
     2\,\zeta^{97} + \zeta^{98}
     \right)  + \\&
  q^5\,\left( \zeta^{-109} + 
     \zeta^{109} \right)  + 
  q^6\,\left( \zeta^{-119} + 
     \zeta^{119} \right)  + 
  q^{11}\,\left( \zeta^{-161} + 
     \zeta^{161} \right)  + \\&
  q^{12}\,\left( {2}
      {\zeta^{-168}} + 2\,\zeta^{168}
     \right)  + 
  q^{13}\,\left( {2}
      {\zeta^{-175}} + 2\,\zeta^{175}
     \right)  + 
  q^{15}\,\left( \zeta^{-188} + 
     \zeta^{188} \right)  + \\&
  q^{16}\,\left( {2}
      {\zeta^{-194}} + 2\,\zeta^{194}
     \right)  + 
  q^{17}\,\left( \zeta^{-200} + 
     \zeta^{200} \right)  + 
  q^{27}\,\left( \zeta^{-252} + 
     \zeta^{252} \right)  + \\&
  q^{29}\,\left( {3}
      {\zeta^{-261}} + 3\,\zeta^{261}
     \right)  + 
  q^{31}\,\left( \zeta^{-270} + 
     \zeta^{270} \right)  + 
  q^{35}\,\left( \zeta^{-287} + 
     \zeta^{287} \right)  + \\&
  q^{37}\,\left( \zeta^{-295} + 
     \zeta^{295} \right)  + 
  q^{67}\,\left( \zeta^{-397} + 
     \zeta^{397} \right)  + 
  q^{74}\,\left( \zeta^{-417} + 
     \zeta^{417} \right)  + \\&
  q^{78}\,\left( {2}
      {\zeta^{-428}} + 2\,\zeta^{428}
     \right)  + 
  q^{79}\,\left( \zeta^{-431} + 
     \zeta^{431} \right)  + 
  q^{85}\,\left( \zeta^{-447} + 
     \zeta^{447} \right)  + \\&
  q^{87}\,\left( {2}
      {\zeta^{-452}} + 2\,\zeta^{452}
     \right)  + 
  q^{94}\,\left( \zeta^{-470} + 
     \zeta^{470} \right)  + 
  q^{101}\,\left( {2}
      {\zeta^{-487}} + 2\,\zeta^{487}
     \right)  + \\&
  q^{106}\,\left( \zeta^{-499} + 
     \zeta^{499} \right)  + 
  q^{109}\,\left( \zeta^{-506} + 
     \zeta^{506} \right)  + 
  q^{116}\,\left( \zeta^{-522} + 
     \zeta^{522} \right)  + \\&
  q^{126}\,\left( {2}
      {\zeta^{-544}} + 2\,\zeta^{544}
     \right)  + 
  q^{133}\,\left( \zeta^{-559} + 
     \zeta^{559} \right)  + 
  q^{134}\,\left( \zeta^{-561} + 
     \zeta^{561} \right)
\\&+ O(q^{148}).  
\end{aligned}
\end{equation*}
The Borcherds product is holomorphic because these singular 
Fourier coefficients are positive. We compute $A=2$, $B=68$, $C=587$, 
and $D_0=1$, so that 
$f= \Borch(\psi) \in  M_2\left( K(587) \right)^{-}= S_2\left( K(587) \right)^{-}$ as desired. 
We have the infinite product expansion 
$$
f\smtwomat{\tau}zz\omega = 
q^2 \zeta^{68} \xi^{587} 
\prod_{\substack{ n,r,m \in \Z:\, m \ge 0, \text{\rm\ if $m=0$ then $n \ge 0$} \\  \text{\rm and if $m=n=0$ then $r < 0$. } }}
\left( 1 - q^n \zeta^r \xi^{587m} \right)^{c(nm,r; \psi)}
$$

\section{First Version of Main Result}

We formalize the considerations that made the example of the previous section work, 
and isolate a condition that guarantees the existence of an 
antisymmetric, meromorphic, paramodular Borcherds product.  
This is the first version of our main result.  

\begin{thm}
\label{mr1}
Fix $m, \ell \in \N$ with $k=24-\ell \ge 0$, and 
 fix $ c \in \N^{\ell} $.  
 Assume that $d \in \N^{\ell}$ satisfies 
 \begin{equation}
\label{BR}
m \prod_{j=1}^{\ell} \dfrac{\zeta^{\frac12 c_j d_j}- \zeta^{-\frac12 c_j d_j}}{\zeta^{\frac12 d_j}- \zeta^{-\frac12 d_j}}= 
 \sigma_2(d) + (2k-1)\sigma_1(d) + 2k^2-3k+\ell . 
\end{equation}
 in $\Z[\zeta^{1/2}, \zeta^{-1/2}]$, 
 where $r_j = \zeta^{d_j} + \zeta^{-d_j}$, 
 and the two symmetric functions are given by 
  $\sigma_2(d) = \sum_{1 \le i < j \le \ell} r_i r_j$ 
 and $\sigma_1(d) = \sum_{1 \le i \le \ell} r_i $. 
 
Then there exists a meromorphic, antisymmetric 
 Borcherds product 
 $$
	 \Borch(\psi)={\tilde \phi \/}\exp\left( - \Grit({ \psi })\right) \in M_k\mero\left( K(N) \right)^{\epsilon},
 $$
 where $\epsilon =(-1)^{k+1}$, 
 $N= \frac12 \sum_{j=1}^{\ell} d_j^2 \in \N$, 
 $\phi=\TB_k(d_1, \ldots, d_{\ell})$, and 
$\Xi=\TB_k(c_1d_1, \ldots, c_{\ell}d_{\ell})$, and 
$$
\psi =\dfrac{\phi\vert V_2 -m \Xi}{\phi} \in J_{0,N}\wh(\Z).  
$$ 
\end{thm}

\begin{lm}
\label{lemma2}
With the assumptions of Theorem~\ref{mr1}, we have 
\begin{align*}
& m \prod_{j=1}^{\ell} c_j = 1080; \quad 
 \sum_{j=1}^{\ell} c_j^2 d_j^2 = 2 \sum_{j=1}^{\ell} d_j^2; \quad 
 \sum_{j=1}^{\ell} d_j^2 \in 2 \Z.
\end{align*}
\end{lm}
\begin{proof}
We let $\zeta \to 1$ in equation~{(\ref{BR})}; 
so $r_j \to 2$, $\sigma_1(d) \to 2\ell$, and $\sigma_2(d) \to 4\binom{\ell}{2}$.  
This gives us 
\begin{align*}
 m \prod_{j=1}^{\ell} c_j &=   4\binom{\ell}{2} +(2k-1)2\ell + 2k^2 -3k +\ell  \\
 &= 2(\ell + k)^2 - 3(\ell + k) = 2(24)^2-3(24) =1080.  
\end{align*}
Both sides of equation~{(\ref{BR})} are zero if we differentiate once 
with respect to $\zeta$ and let $\zeta \to 1$.  We 
differentiate twice 
with respect to $\zeta$ and let $\zeta \to 1$.  
The left hand side gives us 
$$
\left( m \prod_{j=1}^{\ell} c_j  \right)
\sum_{j=1}^{\ell} \frac1{12} d_j^2(c_j^2-1) = 90 \sum_{j=1}^{\ell}  d_j^2(c_j^2-1) .
$$
The right hand side gves us
$$
\sigma_2(d)''\vert_{\zeta = 1} + \sigma_1(d)''\vert_{\zeta=1}
=4(\ell-1) \sum_{j=1}^{\ell} d_j^2 + (2k-1)2  \sum_{j=1}^{\ell} d_j^2
=90  \sum_{j=1}^{\ell} d_j^2.
$$
Thus we have 
$ \sum_{j=1}^{\ell} c_j^2 d_j^2 = 2 \sum_{j=1}^{\ell} d_j^2$ as asserted.  
For the final assertion, we note that 
$$
m \prod_{j=1}^{\ell} \dfrac{\zeta^{\frac12 c_j d_j}- \zeta^{-\frac12 c_j d_j}}{\zeta^{\frac12 d_j}- \zeta^{-\frac12 d_j}}= 
m \zeta^{\sum_{j=1}^{\ell} \frac12 d_j(1-c_j)  }
\prod_{j=1}^{\ell}\left( \sum_{i=0}^{c_j-1} \zeta^{i d_j} \right).
$$
From our assumption that
this equals $ \sigma_2(d) + (2k-1)\sigma_1(d) + 2k^2-3k+\ell$, which 
is in $\Z[\zeta, \zeta\inv]$, we obtain   
$ \sum_{j=1}^{\ell}  d_j(1-c_j) \in 2\Z$ and $ \sum_{j=1}^{\ell}  d_j \equiv  \sum_{j=1}^{\ell}  c_jd_j \mod 2$.  
Thus we have 
$$
\sum_{j=1}^{\ell}  d_j^2 \equiv 
\sum_{j=1}^{\ell}  d_j \equiv 
 \sum_{j=1}^{\ell} c_j d_j \equiv 
  \sum_{j=1}^{\ell} c_j^2 d_j^2 =
  2 \sum_{j=1}^{\ell}  d_j^2 \equiv 0 \mod 2.  \qedhere
$$
\end{proof}

{\it Proof of Theorem~\ref{mr1}.}   
By the Lemma, $N= \frac12 \sum_{j=1}^{\ell} d_j^2$ is integral.  
From the shape of the theta block we have $\phi \in J_{k,N}\weak(2)$.  
We expand $\phi$ out to order~$q^4$.  
\begin{align*}
\phi = &\BTB(d) q^2  \left( 1-  \right.
(2k+ \sigma_1(d)) q  \\
+&\left( \sigma_2(d) + (2k-1)\sigma_1(d) + 2k^2-3k+\ell \right) q^2 +O(q^3) \left. \right) .
\end{align*}
The action of $V_2$ gives us the expansion of $\phi \vert V_2 \in J_{k, 2N}\weak$.  
\begin{align*}
\phi | V_2 = &\BTB(d) q \left( 
1+ \left( \sigma_2(d) + (2k-1)\sigma_1(d) + 2k^2-3k+\ell \right) q +O(q^2) \right) 
\end{align*}
Thus, by the first method,  we have $\frac{\phi \vert V_2}{\phi }  \in J_{0,N}\wh(\Z)$ and 
\begin{align*}
\frac{\phi | V_2}{\phi} = &\frac{1}{q}  +  
\left( {\sigma_2(d) + 2k\sigma_1(d) + 2k^2-k+\ell} \right)  +O(q).
\end{align*}

We turn our attention to the theta block 
$\Xi = \TB_k(c*d)$, 
whose index is 
$\frac12 \sum_{j=1}^{\ell} c_j^2 d_j^2 = \sum_{j=1}^{\ell} d_j^2 =2N$ by Lemma~\ref{lemma2}.  
From the shape of $\Xi$ we have $\Xi \in J_{k, 2N}\weak(2)$ and 
$$
\Xi = \BTB(c*d) q^2 \left( 1 + O(q) \right).  
$$
The theta block $\Xi$ is a strict inflation of $\phi$, 
so that $\BTB(d)$ divides $\BTB(c*d)$ in $\Z[\zeta, \zeta\inv]$ and $\Xi/\phi \in J_{0,N}\wh(\Z)$.  
The $q$-expansion begins
$$
\frac{\Xi}{\phi} = \frac{\BTB(c*d)}{\BTB(d)}\left( 1 + O(q) \right). 
$$

The next step is the crucial one.  
Equation~{(\ref{BR})} says exactly that 
$$
m
 \frac{\BTB(c*d)}{\BTB(d)}
 = \sigma_2(d) + (2k-1)\sigma_1(d) + 2k^2 -3k +\ell.  
$$
The $q$-expansion of 
$\psi = ( \phi \vert V_2 -m \Xi)/\phi \in J_{0,N}\wh(\Z)$ is thus 
\begin{align*}
\psi =\frac1{q} &+ 
(\sigma_2(d) + 2k \sigma_1(d) + 2k^2 - k + \ell) \\
&-(\sigma_2(d) + (2k-1) \sigma_1(d) + 2k^2 -3k + \ell) 
+ O(q)  \\
= \frac1{q} &+\sigma_1(d) + 2k + O(q).  
\end{align*}
The paramodular form $\Borch(\psi)$ exists by Theorem~\ref{BP}.  
The $q^0$-term of $\psi$ is the germ of the theta block~$\phi$, 
therefore the weight of $\Borch(\psi)$ is~$k$, 
$A=\frac1{24}(2k+2\ell)=2$, $ B =\frac12 \sum_{j=1}^{\ell} d_j$, and 
$C= \frac12 \sum_{j=1}^{\ell} d_j^2 =N$.  
We also compute $D_0 = \sum_{n \in \Z, n < 0} \sigma_1(-n)c(n,0; \psi) =  \sigma_1(1)c(-1,0; \psi) = 1$.  
The character of $\Borch(\psi)$ is trivial because $A$ and $B$ are integral.  
$\Borch(\psi)$ is antisymmetric because $D_0$ is odd, implying that $(-1)^k \epsilon =-1$.  
This verifies that $\Borch(\psi) \in M_k\mero\left( K(N) \right)^{\epsilon}$.  $\qed$

\section{Second Version of Main Result}

We reformulate equation~{(\ref{BR})} as a diophantine equation and give a 
second version of our main result. 
We want the variables $d_j$ to satisfy polynomial equations so that 
our meromorphic Borcherds products  correspond to integral 
points on an algebraic set.  
We first 
reduce equation~{(\ref{BR})} to the special case where $\ell=24$, $k=0$, and  $m=1$:
 \begin{equation}
\label{BR24}
 \prod_{j=1}^{24} \dfrac{\zeta^{\frac12 c_j d_j}- \zeta^{-\frac12 c_j d_j}}{\zeta^{\frac12 d_j}- \zeta^{-\frac12 d_j}}= 
 \sigma_2(d) -\sigma_1(d) + 24 ,
\end{equation}
except that we now allow $d \in \Z^{24}$ and interpret 
$
 \dfrac{\zeta^{\frac12 c_j d_j}- \zeta^{-\frac12 c_j d_j}}{\zeta^{\frac12 d_j}- \zeta^{-\frac12 d_j}}
$
as its limiting value $c_j$ when $d_j=0$.  
\smallskip

Fix $c \in \N^{24}$.  
Suppose that $d \in \Z^{24}$ satisfies equation~{(\ref{BR24})} and has~$k$ zero entries.   
For $k +\ell=24$, let $ {\bar d} \in \N^{\ell}$ be given by the nonzero $\vert d_j \vert$ 
in some order.  Let $ {\bar c} \in \N^{\ell}$ be the corresponding entries of~$c$.  
If we set $m = \prod_{j: d_j=0}c_j$ then equation~{(\ref{BR24})} becomes 
\begin{align*}
m &\prod_{j=1}^{\ell} \dfrac{\zeta^{\frac12 {\bar c}_j {\bar d}_j}- \zeta^{-\frac12 {\bar c}_j {\bar d}_j}}
{\zeta^{\frac12 {\bar d}_j}- \zeta^{-\frac12 {\bar d}_j}} \\
=& 4\binom{k}{2} + 2 \binom{k}{1} \sigma_1({\bar d}) + \sigma_2({\bar d}) 
-\left( 2k +  \sigma_1({\bar d}) \right) +24  \\
=& \sigma_2({\bar d}) + (2k-1)\sigma_1({\bar d}) + 2k^2-3k+\ell , 
\end{align*}
which is the general form of equation~{(\ref{BR})} for $ {\bar d} \in \N^{\ell}$ with respect to ${\bar c} \in \N^{\ell}$.  
Thus, by allowing $d $ to have zero entries we may take $ d \in \Z^{24}$ 
as the general case.  Next, we substitute $\zeta = e^{i z}$ in equation~{(\ref{BR24})} and 
reformulate equation~{(\ref{BR24})} as a diophantine problem.  
%
\begin{lm}
Fix $c \in \N^{24}$ with $\prod_{j=1}^{24} c_j =1080$.  Then 
$$
\{ d \in \Z^{24}: \text{equation~{(\ref{BR24})} holds} \}
=
\{ d \in \Z^{24}: \text{equation~{(\ref{wow})} holds} \}
$$
where equation~{(\ref{wow})} means the equality of formal power series 
\begin{multline}
\label{wow}
\exp\left( \sum_{n \in \N}  \frac{(-1)^{n}}{(2n)!} \zeta(1-2n) \sum_{j=1}^{24}   (1- c_j^{2n}) d_j^{2n}  z^{2n} \right) = 1+  \\
\frac{1}{540} \sum_{n \in \N} 
 \frac{(-1)^{n}}{(2n)!} 
 \left( \sum_{1 \le i < j \le 24} \left[ (d_i+d_j)^{2n}+(d_i-d_j)^{2n} \right]   - \sum_{j=1}^{24} d_j^{2n}  \right)  
 z^{2n} 
\end{multline}
\end{lm}
\begin{proof}
Setting $\zeta = e^{i z}$, we have $r_j = \zeta^{d_j} + \zeta^{-d_j} = 2\cos(d_jz)$; also, 
$r_ir_j =4\cos(d_i z)\cos(d_j z) = 2 \cos((d_i+d_j)z) + 2\cos((d_i-d_j)z)$. It is straightforward 
to check that 
\begin{align*}
& \sigma_2(d) -\sigma_1(d) + 24= 1080 
 +  \\
 2 \sum_{n \in \N} &
 \frac{(-1)^{n}}{(2n)!} 
 \left( \sum_{1 \le i < j \le 24} \left[ (d_i+d_j)^{2n}+(d_i-d_j)^{2n} \right]  - \sum_{j=1}^{24} d_j^{2n}  \right)  
 z^{2n} ,
\end{align*}
an equation whose series  converge for all $z \in \C$.  

If we separate out the factors in equation~{(\ref{BR24})} with $d_j=0$, we obtain
$$
 \prod_{j=1}^{24} \dfrac{\zeta^{\frac12 c_j d_j}- \zeta^{-\frac12 c_j d_j}}
 {\zeta^{\frac12 d_j}- \zeta^{-\frac12 d_j}}
 =
 \left( \prod_{j=1}^{24} c_j \right)
\dfrac{\prod_{j: d_j \ne 0}  \frac{\sin(\frac12 c_jd_j z)}{\frac12 c_j d_j z}  }
{\prod_{j: d_j \ne 0}  \frac{\sin(\frac12 d_j z)}{\frac12 d_j z}}.
$$
The following series converges for $\vert z \vert < \pi$.
$$
\ln \left(\dfrac{\sin(z)}{z} \right)
= \sum_{n=1}^{\infty} \ln \left( 1- \dfrac{z^2}{\pi^2 n^2} \right)
=  \sum_{n=1}^{\infty} \dfrac{(-1)^{n+1}}{(2n)!} \zeta(1-2n) (2z)^{2n}, 
$$
where $\zeta(1-2n)$ denotes the value of the Riemann zeta function.  
Using this we have 
\begin{align*}
& \prod_{j=1}^{24} \dfrac{\zeta^{\frac12 c_j d_j}- \zeta^{-\frac12 c_j d_j}}
 {\zeta^{\frac12 d_j}- \zeta^{-\frac12 d_j}}  \\
 =
  &\left(\prod_{j=1}^{24} c_j \right)
 \exp
\left(\sum_{j: d_j \ne 0} \ln \left( \frac{\sin(\frac12 c_jd_j z)}{\frac12 c_j d_j z}  \right)- \ln \left( \frac{\sin(\frac12 d_j z)}{\frac12  d_j z}  \right)\right)
\\
= 
&1080 
\exp\left( \sum_{n \in \N}  \frac{(-1)^{n}}{(2n)!} \zeta(1-2n) \sum_{j=1}^{24}   (d_j^{2n}- c_j^{2n} d_j^{2n})  z^{2n} \right)
. 
\end{align*}
Reinserting the $d_j$ that are zero does not change the value of the last sum. 
The equality of formal series asserted by equation~{(\ref{wow})} converges at least for 
$\vert z \vert < \pi$, and so equation~{(\ref{wow})} is equivalent to the equality of Laurent polynomials 
in equation~{(\ref{BR24})}.  
\end{proof}
\begin{df}
Take $c \in \N^{24}$ with $\prod_{j=1}^{24} c_j =1080$.  
Define the following algebraic set 
$
A_c = \{ d \in \C^{24}:  \text{equation~{(\ref{wow})} holds} \}.
$
\end{df}
Note that $A_c \subseteq \C^{24}$  is defined by a countable set of homogeneous polynomials, 
one for each positive even degree.  
The algebraic set $A_c$ is invariant under the group that changes the sign of each entry 
and under the permutations of the $24$~indices that fix $c$.  The first few of these 
equations are:
\begin{align*}
&\sum_{j=1}^{24}  c_j^2  d_j^2  = 2 
\sum_{j=1}^{24}    d_j^2. \quad \text{($z^2$ term)} \\
&
\sum_{j=1}^{24}    46d_j^4+6c_j^4d_j^4 = 
\sum_{i,j=1}^{24} [15(c_i^2-1)(c_j^2-1)-8]d_i^2d_j^2. \quad \text{($z^4$ term)} \\
&\sum_{j=1}^{24}  (128-16c_j^6)d_j^6 +
\sum_{i,j=1}^{24}
\left( 224 + 42(1-c_i^2)(1-c_j^4) \right) d_i^2 d_j^4 + \\
&\sum_{i,j,k=1}^{24} 
35(1-c_i^2)(1-c_j^2)(1-c_k^2)d_i^2d_j^2d_k^2 =0. 
\quad \text{($z^6$ term)}
\end{align*}
The second version of the main result is formulated in terms of the algebraic set $A_c$.  
\begin{thm}
\label{mr2}
Take $c \in \N^{24}$ with $\prod_{j=1}^{24} c_j =1080$. 
Every nontrivial  integral point $d \in A_c$ 
corresponds to an antisymmetric meromorphic paramodular Borcherds product as follows. 
By taking the absolute value of each entry, we may assume that  $d \in A_c$ 
has nonnegative entries.  
Let $k$ be the number of zero entries in~$d$, and set $\epsilon= (-1)^{k+1}$.  
The number  $N= \frac12 \sum_{j=1}^{24} d_j^2 $ is intergal.  
Set $m= \prod_{j: d_j=0} c_j$.  
We have
 $
 \Borch(\psi) 
 = {\tilde \phi} \exp\left( -\Grit(\psi) \right)
 \in M_k\mero\left( K(N) \right)^{\epsilon},
 $
where 
$
\psi =\dfrac{\phi\vert V_2 -m \Xi}{\phi} \in J_{0,N}\wh(\Z),  
$ 
 $\phi=\eta^{2k} \prod_{j: d_j \ne 0}(\vartheta_{d_j}/\eta) \in J_{k,N}\weak(2)$, and 
$\Xi=\eta^{2k} \prod_{j: d_j \ne 0}(\vartheta_{c_jd_j}/\eta) \in J_{k, 2N}\weak(2)$. 
%
\end{thm}

\section{Two infinite families of integral points}

For each  $c \in \N^{24}$ with $\prod_{j=1}^{24} c_j =1080$, 
we would like to know the decomposition of the projective algebraic 
set $[ A_c\setminus \{0\} ] \subseteq \Pj^{23}(\C)$ into 
irreducible components.  We do not know this decomposition.  
However, we wrote a program that, given an integral point $d \in A_c$
searches for linear spaces defined over~$\Q$ that contain~$d$ and
also lie in $A_c$. Two infinite families were found in this way for the choice
$$
c=[5,3,3,3,2,2,2,1,1,1,  1,1,1,  1,1,1,  1,1,1,  1,1,1, 1,1] \in \N^{24}.  
$$
For this choice of~$c$, 
$[A_c \setminus \{0\} ] $ contains two projective lines $\Fam 1$ and $\Fam 2$.
\begin{align*}
\Fam1 &= \{
\beta, 2 \beta, \beta + \alpha, \beta - \alpha, 2 \beta + \alpha, 2 \alpha, \alpha, 4 \beta, 
3 \beta + 2 \alpha, 3 \beta + \alpha, 
\\&\qquad 
3 \beta, 3 \beta - \alpha, 2 \beta + 2 \alpha,
 2 \beta, 2 \beta - \alpha, 2 \beta - 2 \alpha, \beta + 2 \alpha, \beta + \alpha, \beta,
\\&\qquad 
 \beta - \alpha, 
\beta - 2 \alpha, \alpha,  0,0
: \alpha,\beta\in\Z
\}
\\
\Fam2 &= \{
\beta, \alpha + \beta, \alpha, \alpha - \beta, \alpha + 2 \beta, \alpha - 2 \beta, 2 \beta,
2 \alpha, 2 \alpha + 2 \beta, 2 \alpha + \beta, 
\\&\qquad 
 2 \alpha - \beta, 
  2 \alpha - 2 \beta, \alpha + 3 \beta, \alpha + \beta, \alpha, 4 \beta, \alpha - \beta, 3 \beta, 2 \beta, \alpha - 3 \beta,
\\&\qquad 
 \beta, \beta, 0, 0: \alpha,\beta\in\Z
\}
\end{align*}
Thus the countable number of homogeneous polynomials are consistent 
for this choice of inflation vector~$c$.  
The authors know of no direct argument that makes this  consistency clear.  
In summary,  we have two infinite families of meromorphic 
antisymmetric Borcherds products with weights bounded by $k \le 23$.  
It is interesting to note that the original weight two example for~$N=587$ 
has a different inflation vector~$c$ and 
is not on either of these families.

\section{Examples}

We are especially interested in holomorphic Borcherds products.  
  A direct search through the two infinite families found the 
{\it holomorphic\/} antisymmetric  paramodular Borcherds products 
 listed in Table~1. We now explain how to read Table~$1$.  \smallskip

  Fix  $ c=[5,3,3,3,2,2,2,1,1,1,  1,1,1,  1,1,1,  1,1,1,  1,1,1, 1,1]$.  
For the 24 integers given by $\Fam1(\alpha, \beta)$ or  $\Fam2(\alpha, \beta)$,  
let $d \in \Z^{24}$ be the vector determined by 
their absolute values in 
the given ordering.  
Let~$k$ be the number of zero entries in~$d$, and let 
and $j_1, \ldots, j_{\ell}$ be the indices of the nonzero entries.  
For $N =\frac12 \sum_{j=1}^{24} d_j^2$,  and $m= \prod_{j: d_j=0}c_{j}$, define 
\begin{align*}
\phi  &=\TB{}_k( d_{j_1}, \ldots,  d_{j_{\ell} }) \in J_{k,N}\weak; \quad 
\Xi   =\TB{}_k(c_{j_1}d_{j_1}, \ldots, c_{j_{\ell} }d_{j_{\ell} }) \in J_{k,2N}\weak, \\  
\psi   &=\dfrac{\phi\vert V_2 -m \Xi}{\phi} \in J_{0,N}\wh(\Z).  
\end{align*}
Table~$1$ gives antisymmetric 
$\Borch(\psi) \in S_k\left( K(N) \right)^{\epsilon}$ for $\epsilon = (-1)^{k+1}$.  
Note that $\phi$ is the leading Fourier-Jacobi coefficient of $\Borch(\psi)$.  
\bigskip

Table 1. Antisymmetric  Borcherds products 
in $S_k\left(  K(N) \right)^{\epsilon}$.  
\begin{center}
\renewcommand{\arraystretch}{1.5}
\begin{longtable}{|r|r|r|r|r|r|}
\hline
$k$ & $N$ & $m$   & $(\alpha,\beta) \text{ for }\Fam1$  & $(\alpha,\beta) \text{ for }\Fam2$ & $\epsilon$ \\
\hline
2 & $587$ & $1$ 
 &  $\text{ }$ & \text{ } & -1 \\
\hline
2 & $713$ & $1$ 
  & $(1,4) \text{ or } (-5,3) $  &  & -1 \\
\hline
2 & $893$ & $1$ 
  & $ (5,3) $  &  & -1 \\
\hline
3 & $122$ & $1$ 
  & $ (2,1) $  &  & +1 \\
\hline
3 & $167$ & $1$ 
  & $\text{ } $  &$ (1,2) $  & +1 \\
\hline
3 & $173$ & $1$ 
  & $ (-2,1) $  &$ (3,1) $  & +1 \\
\hline
3 & $197$ & $1$ 
  & $ (1,2) $  &$\text{ } $  & +1 \\
\hline
3 & $213$ & $1$ 
  & $ (3,1) $  &$\text{ } $  & +1 \\
\hline
3 & $285$ & $1$ 
  & $ (-3,2)  $  &$\text{ } $  & +1 \\
\hline
5 & $38$ & $3$ 
  & $\text{ } $  &$ (0,1)  $  & +1 \\
\hline
5 & $42$ & $4$ 
  & $ (0,1) $  &$\text{ } $  & +1 \\
\hline
5 & $53$ & $3$ 
  & $ (-1,1) $  &$  (1,1) $  & +1 \\
\hline
5 & $65$ & $3$ 
  &  $ (1,1) $  & \text{ }  & +1 \\
\hline
8 & $17$ & $15$ 
  & $ (1,0) $  &$\text{ } $  & -1 \\
\hline
9 & $15$ & $10$ 
  & $\text{ } $  &$ (1,0) $  & +1 \\
\hline
\end{longtable}
\end{center}

We make some concluding remarks about these examples.  
Like $587$, $N=713$ and~$893$  conjecturally show the modularity of 
known abelian surfaces defined over~$\Q$ of rank one and conductor~$N$.  
According to \cite{BK}, equations of hyperelliptic curves whose Jacobians give these 
abelian surfaces are $y^2= x^6-2x^5 +x^4 +2x^3 +2x^2 -4x +1$ for~$N=713$ and 
$y^2 = x^6 -2x^4 -2x^3 -3x^2 -2x +1$ for~$N=893$. 
\medskip

The Siegel modular threefold $K(t)\setminus \Bbb H_2$
is the moduli space of $(1,t)$-polarized abelian surfaces because 
$K(t)$  is isomorphic to the integral symplectic group of the symplectic form with elementary divisors $(1,t)$.
The paramodular group $K(t)$ has the maximal extension $K(t)^*$
in ${\rm Sp}_2(\Bbb R)$ of order $2^{\nu(t)}$ where $\nu(t)$ is the number of  prime divisors of $t$, see \cite{GritHulek0}. 
In \cite[Theorem 1.5]{GritHulek0} it was proved that the modular variety   
$K(t)^*\setminus \Bbb H_2$  can be considered as the moduli space of Kummer surfaces associated to $(1,t)$-polarized abelian surfaces. 
It was noted in \cite{Grit2} that the moduli space of polarized abelian surfaces might have trivial geometric genus only 
for twenty exceptional polarizations
$$
t=1,\ldots,12, \ 14,\ 15,\ 16,\ 18,\  20,\ 24, \ 30, \ 36. 
$$
It is now known \cite{BPY} that $\dim S_3\left( K(t) \right) =0$ for these~$t$.  
For the moduli spaces of polarized Kummer surfaces we expect a rather 
long  list of exceptional polarizations.
One result in this direction is that 
for $t=21$ the space $K(t)^*\setminus \Bbb H_2$ is uniruled, see  
\cite{GritHulek2}.  

Using our method we can construct the first canonical differential forms
on $K(t)^*\setminus \Bbb H_2$. According to the
Freitag criterion (see \cite[Hilfsatz 3.2.1]{freitag83}) for any smooth compactification 
$\overline{K(t)^*\setminus \Bbb H_2}$, we have  
$$
h^{3,0}(\overline{K(t)^*\setminus \Bbb H_2})=
{\rm dim\ } S_3(K(t)^*)
$$
where $S_3(K(t)^*)$ denotes the space of antisymmetric cusp forms having 
all Atkin-Lehner signs equal to~$+1$. 
We know $\dim S_3\left( K(t)^{*} \right) =0$ for $t \le 40$, see \cite{BPY}. 

For weight three, the first antisymmetric form we know of is in $S_3\left( K(122) \right)^{+}$  
but it is an oldform, in the sense of \cite{RS}, and comes from a newform in $S_3\left( K(61) \right)^{-}$. 
As such,  the oldform has Atkin-Lehner signs of $-1$ at both $2$ and $61$. 
The first  nontrivial $S_3(K(t)^*)$ that we know of is for the prime $t=167$. 
The cases $173,\  197,\  213$, and $285$ also have this property, 
with $213$ having $+1$ Atkin-Lehner signs at $3$ and $71$, and with 
$285$  having $+1$ Atkin-Lehner signs at $3$, $5$, and $19$.

\begin{cor} 
The moduli space $K(t)^*\setminus \Bbb H_2$ 
of Kummer surfaces associated to $(1,t)$-polarized abelian surfaces
has positive geometric genus if 
$t=167$, $173$, $197$, $213$, and $285$. In particular,
$H^3(K(t)^*, \C)$ is nontrivial for these $t$.
\end{cor}
 
Contributions to the cohomology $H^5(\Gamma_0(N), \C)$   
studied by  Ash, Gunnells and McConnell 
can also be seen in Table~4 of \cite{AGM} for the primes $N=167, 173$ and $197$.


\end{document}